\documentclass{article}

\usepackage{amssymb}
\usepackage{amsmath}
\usepackage{amsfonts}
\usepackage{amsthm}

\usepackage{a4wide}

\newtheorem{theorem}{Theorem}

\newtheorem{lemma}[theorem]{Lemma}
\newtheorem{corollary}[theorem]{Corollary}
\newtheorem*{problemP}{Problem $\mathbf{(P)}$}
\newtheorem{definition}[theorem]{Definition}
\newtheorem{remark}[theorem]{Remark}
\theoremstyle{definition}
\newtheorem{example}[theorem]{Example}

\begin{document}

\title{Higher-order variational problems of Herglotz type\thanks{This 
is a preprint of a paper whose final and definite form will appear 
in \emph{Vietnam Journal of Mathematics}, ISSN: 2305-221X (Print) 2305-2228 (Online). 
Paper submitted 13-June-2013; revised  11-Sept-2013; accepted for publication 24-Sept-2013.}}

\author{Sim\~{a}o P. S. Santos\\ {\tt simsantos@gmail.com}
\and Nat\'{a}lia Martins\\  {\tt natalia@ua.pt}
\and Delfim F. M. Torres\\ {\tt delfim@ua.pt}}

\date{CIDMA--Center for Research and Development in Mathematics and Applications\\
Department of Mathematics, University of Aveiro, 3810-193 Aveiro, Portugal}

\maketitle


\begin{abstract}
We obtain a generalized Euler--Lagrange differential equation and transversality optimality conditions
for Herglotz-type higher-order variational problems. Illustrative examples of the new results are given.

\bigskip


\bigskip

\end{abstract}


\section{Introduction}

The generalized variational calculus proposed by Herglotz
\cite{Guenther1996,Herglotz1930} deals with an initial value problem
\begin{gather}
\label{eq_Herg}
\dot{z}(t)=L\left(t,x(t),\dot{x}(t),z(t)\right), \quad t \in [a,b],\\
z(a)=\gamma, \quad \gamma \in \mathbb{R}, \label{eq_initialvalue_Herg}
\end{gather}
and consists in determining trajectories $x$ that extremize (minimize or maximize) the value $z(b)$.
Observe that \eqref{eq_Herg} represents a family of differential equations: for each function $x$
a different differential equation arises. Therefore, $z$ depends on $x$, a fact that
can be made explicit by writing $z(t, x(t), \dot{x}(t))$ or $z[x;t]$,
but for brevity and convenience of notation it is usual to write simply $z(t)$.
The problem reduces to the classical fundamental problem
of the calculus of variations (see, e.g., \cite{GelfandFomin}) if the Lagrangian $L$
does not depend on the variable $z$: if
\begin{gather*}
\dot{z}(t)=L(t,x(t),\dot{x}(t)), \quad t \in [a,b],\\
z(a)=\gamma, \quad \gamma \in \mathbb{R},
\end{gather*}
then we obtain the classical variational problem
\begin{equation}
\label{eq:class:Funct}
z(b)=\int_a^b \tilde{L}(t,x(t),\dot{x}(t))dt \longrightarrow \textrm{extr},
\end{equation}
where
$$
\tilde{L}(t,x,\dot{x})=L(t,x,\dot{x})+\frac{\gamma}{b-a}.
$$
Herglotz proved that a necessary condition for a trajectory $x$
to be an extremizer of the generalized variational problem $z(b) \rightarrow \textrm{extr}$
subject to \eqref{eq_Herg}--\eqref{eq_initialvalue_Herg} is given by
\begin{equation}
\label{eq:gen:EL:eq}
\frac{\partial L}{\partial x}\left(t,x(t),\dot{x}(t),z(t)\right)
-\frac{d}{dt}\frac{\partial L}{\partial \dot{x}}\left(t,x(t),\dot{x}(t),z(t)\right)
+\frac{\partial L}{\partial z}\left(t,x(t),\dot{x}(t),z(t)\right)
\frac{\partial L}{\partial \dot{x}}\left(t,x(t),\dot{x}(t),z(t)\right) = 0,
\end{equation}
$t \in [a,b]$. This equation is known as the generalized
Euler--Lagrange equation \cite{MR1391230,Guenther1996,MR1738100}.
Observe that for the classical problem of the calculus of variations \eqref{eq:class:Funct}
one has $\frac{\partial L}{\partial z}=0$, and the differential equation \eqref{eq:gen:EL:eq}
reduces to the classical Euler--Lagrange equation:
$$
\frac{\partial L}{\partial x}\left(t,x(t),\dot{x}(t)\right)
-\frac{d}{dt}\frac{\partial L}{\partial \dot{x}}\left(t,x(t),\dot{x}(t)\right) =0.
$$
The variational problem of Herglotz was the basis of the Ph.D. thesis \cite{Georgieva2001}.
The main goal of this thesis, done under supervision of Ronald B. Guenther,
was to generalize the well known Noether's theorems (see, e.g., \cite{MR2098297})
to problems of Herglotz type \cite{Georgieva2002,Georgieva2005,Georgieva2003}.
As reported in \cite{Georgieva2002,Georgieva2005}, unlike the classical variational principle,
the variational principle of Herglotz gives a variational description of nonconservative processes,
even when the Lagrangian is autonomous. For the importance to include nonconservativism
in the calculus of variations, we refer the reader to the recent book \cite{book:FCV}.

In this paper we generalize Herglotz's problem by considering
the following higher-order variational problem.

\begin{problemP}
Determine the trajectories $x\in C^{2n}([a,b], \mathbb{R})$
that extremize the value of the functional $z[x;b]$,
$$
z(b) \longrightarrow \textrm{extr},
$$
where $z$ satisfies the differential equation
\begin{equation}
\label{eq_main}
\dot{z}(t)=L\left(t,x(t),\dot{x}(t),\ldots,x^{(n)}(t),z(t)\right), \quad t \in [a,b],
\end{equation}
subject to the initial condition
\begin{equation}
\label{eq_initialvalue}
z(a)=\gamma,
\end{equation}
where $\gamma$ is a fixed real number.
The Lagrangian $L$ is assumed to satisfy the following hypotheses:
\begin{description}
\item[$(H1)$] $L$ is a $C^1(\mathbb{R}^{n+3}, \mathbb{R})$ function;
\item[$(H2)$] functions  $\displaystyle t \mapsto \frac{\partial L}{\partial x^{(j)}}\left(t,
x(t), \dot{x}(t), \ldots, x^{(n)}(t), z(t)\right)$  and $\displaystyle t \mapsto
\frac{\partial L}{\partial z}\left(t, x(t), \dot{x}(t), \ldots, x^{(n)}(t), z(t)\right)$,
$j=0, \ldots, n$, are differentiable up to order $n$ for any admissible trajectory $x$.
\end{description}
\end{problemP}

Clearly, problem $(P)$ generalizes the classical higher-order variational problem:
if the Lagrangian $L$ is independent of $z$, then
\begin{gather*}
\dot{z}(t)=L\left(t,x(t),\dot{x}(t),\ldots, x^{(n)}(t)\right), \quad t \in [a,b],\\
z(a)=\gamma, \quad \gamma \in \mathbb{R},
\end{gather*}
which implies that the problem under consideration is
$$
z(b)=\int_a^b \tilde{L}\left(t,x(t),\dot{x}(t),\ldots,x^{(n)}(t)\right)dt
\longrightarrow \textrm{extr},
$$
where
$$
\tilde{L}\left(t,x,\dot{x},\ldots,x^{(n)}\right)
=L\left(t,x,\dot{x},\ldots,x^{(n)}\right)+\frac{\gamma}{b-a}.
$$

The paper is organized as follows. In Section~\ref{sec:prelim}
we recall the necessary results from the classical calculus of variations.
Our results are then proved in Section~\ref{sec:mr}:
in Section~\ref{section3.1_E-L} we obtain the generalized
Euler--Lagrange equation for problem $(P)$
in the class of functions $x \in C^{2n}([a,b],\mathbb{R})$
satisfying given boundary conditions
\begin{equation}
\label{main_constrains}
\begin{array}{c}
x(a)=\alpha_0,  \, \ldots, \, x^{(n-1)}(a)=\alpha_{n-1},\\
x(b)=\beta_0,  \, \ldots, \, x^{(n-1)}(b)=\beta_{n-1},
\end{array}
\end{equation}
where $\alpha_0$, $\ldots$, $\alpha_{n-1}$, and
$\beta_0$, \ldots, $\beta_{n-1}$ are given real numbers.
The transversality conditions (or natural boundary conditions)
for problem $(P)$ are obtained in Section~\ref{section3.2_ NBC}.
We end with Section~\ref{section4_ex}, presenting some illustrative
examples of application of the new results.

The results of the paper are trivially generalized for the case of vector functions
$x: [a,b] \rightarrow \mathbb{R}^m$, $m \in \mathbb{N}$, but for simplicity of presentation
we restrict ourselves to the scalar case. Along the text, we use the standard conventions
$x^{(0)} = \frac{d^0 x}{dt^0} = x$ and $\sum_{k=1}^{j}\Upsilon(k)=0$ whenever $j=0$.


\section{Preliminary results}
\label{sec:prelim}

We recall some results of the classical calculus of variations
that are useful in the sequel.

\begin{definition}
We say that $\eta \in C^{2n}\left([a,b], \mathbb{R}\right)$
is an admissible variation for problem $(P)$ subject to boundary conditions
\eqref{main_constrains} if, and only if,
$\eta(a)=\eta(b)=  \cdots=  \eta^{(n-1)}(a)=\eta^{(n-1)}(b)=0$.
\end{definition}

\begin{lemma}[Higher-order fundamental lemma of the calculus of variations -- cf. \cite{MartinsTorres2009}]
\label{FlemmaCV}
Let $f_0$, $\ldots$, $f_n \in C([a,b], \mathbb{R})$. If
$$
\int_{a}^{b}
\left(\sum_{i=0}^{n}f_i(t) \eta^{(i)}(t) \right) dt=0
$$
for all admissible variations $\eta$ of problem $(P)$
subject to boundary conditions \eqref{main_constrains}, then
$$
\sum_{i=0}^{n} (-1)^i f_i^{(i)}(t) =0,
$$
$t \in [a,b]$.
\end{lemma}

\begin{lemma}[Higher-order integration by parts formulas -- cf. \cite{MartinsTorres2012}]
\label{Integration:byparts:higherorder}
Let $n \in \mathbb{N}$, $a,b \in \mathbb{R}$, $a<b$, and $f,g \in C^n\left([a,b], \mathbb{R}\right)$.
The following $n$ equalities hold:
$$
\int^b_af(t)g^{(i)}(t)dt=\left[ \sum^{i-1}_{k=0}(-1)^kf^{(k)}(t)g^{(i-1-k)}(t)\right]^b_a
+ (-1)^i\int^b_af^{(i)}(t)g(t)dt,
$$
$i=1,\ldots,n$.
\end{lemma}


\section{Main results}
\label{sec:mr}

For simplicity of notation, we introduce the operator
$\langle \cdot, \cdot \rangle_n$, $n \in \mathbb{N}$, defined by
$$
\langle x, z \rangle_n (t):= \left(t,x(t),\dot{x}(t), \ldots, x^{(n)}(t),z(t)\right).
$$


\subsection{Generalized Euler--Lagrange equation}
\label{section3.1_E-L}

The following result gives a necessary condition of Euler--Lagrange type for an admissible
function $x$ to be an extremizer of the functional $z[x;b]$, where $z$
is defined by \eqref{eq_main}, \eqref{eq_initialvalue} and \eqref{main_constrains}.

\begin{theorem}[Generalized higher-order Euler--Lagrange equation]
\label{main_Thm}
If $x$ is a solution to problem $(P)$ subject to the boundary conditions \eqref{main_constrains},
then $x$ satisfies the generalized Euler--Lagrange equation
\begin{equation}
\label{H-O_E-L_equations}
\sum^{n}_{j=0}(-1)^j \frac{d^j}{dt^j} \left(\lambda(t)
\frac{\partial L}{\partial x^{(j)}} \langle x, z \rangle_n (t) \right)=0,
\end{equation}
$t \in [a,b]$, where $\displaystyle \lambda(t):=\displaystyle e^{-\int^t_a\frac{\partial L}{\partial z}
\langle x, z \rangle_n(\theta)d\theta}$.
\end{theorem}

\begin{proof}
Suppose that $x$ is a solution of $(P)$ subject to \eqref{main_constrains},
and let $\eta \in C^{2n}([a,b], \mathbb{R})$ be an admissible variation
such that $\eta^{(n)}(a)=0$. Let $\epsilon$ be an arbitrary real number.
Define $\zeta:[a,b] \rightarrow \mathbb{R}$ by
$$
\zeta(t):=\frac{d}{d\epsilon}z[x+\epsilon \eta; t]\biggm\vert_{\epsilon=0}
=\frac{d}{d\epsilon}z\left(t, x(t)+\epsilon\eta(t), \dot{x}(t)
+\epsilon\dot{\eta}(t),\ldots,x^{(n)}(t)+\epsilon\eta^{(n)}(t)\right)\biggm\vert_{\epsilon=0}.
$$
Obviously, $\zeta(a)=0$. Since $x$ is a minimizer (resp., maximizer), we have
\begin{equation*}
z\left(b, x(b) +\epsilon \eta(b), \dot x(b) + \epsilon \dot\eta(b), \ldots,
x^{(n)}(b) + \epsilon\eta^{(n)}(b)\right)
\geq (\text{resp. } \leq) \, z\left(b, x(b), \dot x(b), \ldots, x^{(n)}(b)\right).
\end{equation*}
Hence, $\zeta(b)=\left.\frac{d}{d\epsilon} z[ x+\epsilon\eta;b]\right|_{\epsilon=0}=0$.
Because
\begin{equation*}
\begin{split}
\dot{\zeta}(t)&=\frac{d}{dt}\frac{d}{d\epsilon}z\left(t, x(t)+\epsilon\eta(t),
\dot{x}(t)+\epsilon\dot{\eta}(t),\ldots,x^{(n)}(t)+\epsilon\eta^{(n)}(t)\right)\biggm\vert_{\epsilon=0}\\
&=\frac{d}{d\epsilon}\frac{d}{dt}z\left(t, x(t)+\epsilon\eta(t), \dot{x}(t)
+\epsilon\dot{\eta}(t),\ldots,x^{(n)}(t)+\epsilon\eta^{(n)}(t)\right)\biggm\vert_{\epsilon=0}\\
&=\frac{d}{d\epsilon} L\langle x+\epsilon \eta, z  \rangle_n (t)\biggm\vert_{\epsilon=0},
\end{split}
\end{equation*}
we conclude that
\begin{equation*}
\begin{split}
\dot{\zeta}(t)&= \sum_{k=0}^{n} \left(
\frac{\partial L}{\partial x^{(k)}}\langle x, z  \rangle_n (t)\eta^{(k)}(t)\right)
+\frac{\partial L}{\partial z}\langle x, z \rangle_n(t)
\frac{d}{d\epsilon}z[x+\epsilon \eta;t] \biggm\vert_{\epsilon=0}\\
&= \sum_{k=0}^{n} \left(
\frac{\partial L}{\partial x^{(k)}}\langle x, z  \rangle_n (t)\eta^{(k)}(t)\right)
+\frac{\partial L}{\partial z}\langle x, z  \rangle_n (t)\zeta(t).
\end{split}
\end{equation*}
Thus, $\zeta$ satisfies a first order linear differential equation
whose solution is found according to
$$
\dot{y}+Py=Q \Leftrightarrow e^{-\int_a^t P(\theta) d\theta}y(t)-y(a)
=\int_a^t{e^{-{\int_a^{\tau}{P(\theta)}d\theta}}Q(\tau) d\tau}.
$$
Therefore,
\begin{equation*}
e^{-\int^t_a{\frac{\partial L}{\partial z}}\langle x, z  \rangle_n (\theta) d\theta}\zeta(t) - \zeta(a)
=\int^t_a e^{-\int^{\tau}_a{\frac{\partial L}{\partial z}\langle x, z  \rangle_n (\theta)}d\theta}\left(
\sum_{j = 0}^{n}\frac{\partial L}{\partial x^{(j)}}\langle x, z  \rangle_n (\tau)\, \eta^{(j)}(\tau)\right)d\tau.
\end{equation*}
Denoting $\lambda(t):=e^{ -\int^t_a\frac{\partial L}{\partial z}\langle x, z  \rangle_n (\theta)d\theta}$, we get
$$
\lambda(t)\zeta(t)-\zeta(a)
=\int^t_a \lambda(\tau) \left(\sum^n_{j=0}\frac{\partial L}{\partial
x^{(j)}}\langle x, z  \rangle_n (\tau)\eta^{(j)}(\tau)\right) d\tau.
$$
In particular, for $t=b$, we have
$$
\lambda(b)\zeta(b)-\zeta(a)=\int^b_a \lambda(\tau) \left(\sum^n_{j=0}\frac{\partial
L}{\partial x^{(j)}}\langle x, z  \rangle_n (\tau)\eta^{(j)}(\tau)\right) d\tau.
$$
Since $\zeta(t)=0$, $t \in \{a,b\}$, the left-hand side of the previous equation vanishes and we get
$$
0=\int^b_a\sum^n_{j=0}\lambda(\tau)\frac{\partial L}{\partial
x^{(j)}}\langle x, z  \rangle_n (\tau)\eta^{(j)}(\tau)d\tau.
$$
Using the higher-order fundamental lemma of the calculus of variations
(Lemma~\ref{FlemmaCV}), we obtain the generalized Euler--Lagrange equation
$$
\sum^{n}_{j=0}(-1)^j\left(\lambda(t) \frac{\partial L}{\partial x^{(j)}}\langle x, z  \rangle_n (t)\right)^{(j)}=0,
$$
$t \in [a,b]$, proving the intended result.
\end{proof}

In order to simplify expressions, and in agreement with
Theorem~\ref{main_Thm}, from now on we use the notation
$\lambda(t):=e^{-\int^t_a\frac{\partial L}{\partial z}\langle x, z  \rangle_n (\theta)d\theta}$.
If $n=1$, then the differential equation of problem $(P)$ reduces to
$\dot{z}(t)=L\left(t,x(t),\dot{x}(t),z(t)\right)$,
which defines the functional $z$ of Herglotz's variational principle.
This principle is a particular case of our Theorem~\ref{main_Thm}
and is given in Corollary~\ref{cor:HP}.

\begin{corollary}[See \cite{Guenther1996,Herglotz1930}]
\label{cor:HP}
Let $z$ be a solution of $\dot{z}(t)=L\left(t,x(t),\dot{x}(t),z(t)\right)$,
$t\in [a,b]$, subject to the boundary conditions $z(a)=\gamma$,
$x(a)=\alpha$ and  $x(b)=\beta$,
where $\gamma$, $\alpha$, and $\beta$ are given real numbers.
If $x$ is an extremizer of functional $z[x;b]$, then $x$ satisfies the differential equation
$$
\frac{\partial L}{\partial x}\left(t, x(t),\dot{x}(t), z(t)\right)
+\frac{\partial L}{\partial z}\left(t, x(t),\dot{x}(t), z(t)\right)
\frac{\partial L}{\partial \dot{x}}\left(t, x(t),\dot{x}(t), z(t)\right)
-\frac{d}{dt}\frac{\partial L}{\partial \dot{x}}\left(t, x(t),\dot{x}(t), z(t)\right) = 0,
$$
$t\in [a,b]$.
\end{corollary}

Our Euler--Lagrange equation \eqref{H-O_E-L_equations} is also a generalization
of the classical Euler--Lagrange equation for higher-order variational problems.

\begin{corollary}[See, e.g., \cite{GelfandFomin}]
Suppose that $x$ is a solution of problem $(P)$ subject to \eqref{main_constrains},
and that the Lagrangian $L$ is independent of $z$.
Then $x$ satisfies the classical higher-order Euler--Lagrange differential equation
\begin{equation}\label{ClassicalE-L}
\sum^{n}_{j=0}(-1)^j  \frac{d^j}{dt^j} \left(
\frac{\partial L}{\partial x^{(j)}}\left(t, x(t),
\ldots, x^{(n)}(t)\right)\right) =0,
\end{equation}
$t\in [a,b]$.
\end{corollary}


\subsection{Generalized natural boundary conditions}
\label{section3.2_ NBC}

We now consider the case when the values of
$x(a)$, $\ldots$, $x^{(n-1)}(a)$,
$x(b)$, $\ldots$, $x^{(n-1)}(b)$,
are not necessarily specified.

\begin{theorem}[Generalized natural boundary conditions]
\label{Thm_Nat_Bound_Cond}
Suppose that $x$ is a solution to problem $(P)$. Then
$x$ satisfies the generalized Euler--Lagrange equation \eqref{H-O_E-L_equations}. Moreover,
\begin{enumerate}

\item If $x^{(k)}(b)$, $k \in \{0,\ldots, n-1\}$, is free, then the natural boundary condition
\begin{equation}
\label{NBC(b)}
\sum_{j=1}^{n-k}(-1)^{j-1} \frac{d^{j-1}}{dt^{j-1}} \left(
\lambda(t)\frac{\partial L}{\partial x^{(k+j)}}\langle x,
z  \rangle_n (t)\right)\Bigg\vert_{t=b}=0
\end{equation}
holds.

\item If $x^{(k)}(a)$,
$k \in \{0,\ldots, n-1\}$, is free, then the natural boundary condition
\begin{equation}
\label{NBC(a)}
\sum_{j=1}^{n-k}(-1)^{j-1} \frac{d^{j-1}}{dt^{j-1}}
\left( \lambda(t)\frac{\partial L}{\partial
x^{(k+j)}}\langle x, z  \rangle_n (t)\right)\Bigg\vert_{t=a}=0
\end{equation}
holds.
\end{enumerate}
\end{theorem}

\begin{proof}
Suppose that $x$ is a solution to problem $(P)$. Let $\eta \in C^{2n}([a,b], \mathbb{R})$
and define the function $\zeta$ just like in the proof of Theorem~\ref{main_Thm}.
From the arbitrariness of $\eta$, and using similar arguments as the ones
in the proof of Theorem~\ref{main_Thm}, we conclude that $x$ satisfies the
generalized Euler--Lagrange equation \eqref{H-O_E-L_equations}.
We now prove \eqref{NBC(b)} (the proof of \eqref{NBC(a)} follows exactly the
same arguments). Suppose that $x^{(k)}(b)$, $k \in \{0,\ldots,n-1\}$, is free.
Define the function
$\zeta(t):=\left.\frac{d}{d\epsilon}z[x+\epsilon \eta; t]\right\vert_{\epsilon=0}$.
Let $J : = \left\{ j \in \{0, \ldots, n-1\} : x^{(j)}(a) \text{ is given}\right\}$.
For any $j \in \{0, \ldots, n-1\}$, if $j \in J$, then $\eta^{(j)}(a)=0$;
otherwise, we restrict ourselves to those functions $\eta$ such that $\eta^{(j)}(a)=0$.
For convenience, we also suppose that $\eta^{(n)}(a)=0$. Using the same arguments
as the ones used in the proof of Theorem~\ref{main_Thm},
we find that $\zeta$ satisfies the first order linear differential equation
$$
\dot{\zeta}(t) = \frac{\partial L}{\partial x}\langle x, z  \rangle_n (t)\eta(t)
+\frac{\partial L}{\partial \dot{x}}\langle x, z  \rangle_n (t)\dot{\eta}(t)
+\cdots+\frac{\partial L}{\partial x^{(n)}}\langle x, z  \rangle_n (t)\eta^{(n)}(t)
+\frac{\partial L}{\partial z}\langle x, z  \rangle_n (t)\zeta(t),
$$
whose solution is found by
$$
\lambda(t)\zeta(t)-\zeta(a)=\int^t_a\sum^n_{j=0}\lambda(\tau)\frac{\partial L}{\partial
x^{(j)}}\langle x, z  \rangle_n (\tau)\eta^{(j)}(\tau)d\tau.
$$
Again, since $\zeta(t)=0$, for $t \in \{a,b\}$, we get
$$
\int^b_a\sum^n_{j=0}\lambda(\tau)\frac{\partial L}{\partial x^{(j)}}\langle x,
z  \rangle_n (\tau)\eta^{(j)}(\tau)d\tau=0
$$
and, therefore,
$$
\int^b_a\lambda(\tau)\frac{\partial L}{\partial x}\langle x,
z  \rangle_n (\tau)\eta(\tau)d\tau + \sum^n_{j=1}\int^b_a\lambda(\tau)
\frac{\partial L}{\partial x^{(j)}}\langle x,
z  \rangle_n (\tau)\eta^{(j)}(\tau)d\tau=0.
$$
Using the higher-order integration by parts formula
(Lemma~\ref{Integration:byparts:higherorder}) in the second parcel we get
\begin{multline*}
\int^b_a\lambda(\tau)\frac{\partial L}{\partial x}\langle x, z  \rangle_n (\tau)\eta(\tau)d\tau\\
+ \sum^{n}_{j=1}\left( \left[\lambda(\tau)
\frac{\partial L}{\partial x^{(j)}}\langle x, z  \rangle_n (\tau)\eta^{(j-1)}(\tau)
+ \sum^{j-1}_{i=1}(-1)^i\left(\lambda(\tau)
\frac{\partial L}{\partial x^{(j)}}\langle x, z
\rangle_n (\tau)\right)^{(i)}\eta^{(j-1-i)}(\tau)\right]^b_a\right.\\
\left.+(-1)^j \int^b_a \left(\lambda(\tau) \frac{\partial L}{\partial x^{(j)}}\langle
x, z  \rangle_n (\tau)\right)^{(j)}\eta(\tau) d\tau \right) =0,
\end{multline*}
which is equivalent to
\begin{multline*}
\int^b_a \sum_{j=0}^{n}(-1)^j  \left(\lambda(\tau)
\frac{\partial L}{\partial x^{(j)}}\langle x, z
\rangle_n (\tau)\right)^{(j)}\eta(\tau) d\tau\\
+ \sum^{n}_{j=1} \left[\lambda(\tau) \frac{\partial L}{\partial x^{(j)}}\langle x,
z  \rangle_n (\tau)\eta^{(j-1)}(\tau) + \sum^{j-1}_{i=1}(-1)^i\left(\lambda(\tau)
\frac{\partial L}{\partial x^{(j)}}\langle x, z
\rangle_n (\tau)\right)^{(i)}\eta^{(j-1-i)}(\tau)\right]^b_a=0.
\end{multline*}
Using the generalized Euler--Lagrange equation
\eqref{H-O_E-L_equations} into the last equation we get
\begin{equation*}
\sum^{n}_{j=1} \left[\lambda(\tau) \frac{\partial L}{\partial
x^{(j)}}\langle x, z \rangle_n (\tau)\eta^{(j-1)}(\tau)
+ \sum^{j-1}_{i=1}(-1)^i\left(\lambda(\tau)
\frac{\partial L}{\partial x^{(j)}}\langle x,
z \rangle_n (\tau)\right)^{(i)}\eta^{(j-1-i)}(\tau)\right]^b_a=0
\end{equation*}
and since $\eta(a)=\dot{\eta}(a)=\cdots=\eta^{(n-1)}(a)=0$, we conclude that
\begin{equation*}
\sum^{n}_{j=1} \left(\lambda(\tau) \frac{\partial L}{\partial
x^{(j)}}\langle x, z \rangle_n (\tau)\eta^{(j-1)}(\tau)
+\sum^{j-1}_{i=1}(-1)^i\left(\lambda(\tau) \frac{\partial L}{\partial x^{(j)}}\langle x,
z \rangle_n (\tau)\right)^{(i)}\eta^{(j-1-i)}(\tau)\right)\Bigg\vert_{\tau=b} =0.
\end{equation*}
This equation is equivalent to
$$
\sum_{i=0}^{n-1}\left( \sum_{j=1}^{n-i}(-1)^{j-1}\left( \lambda(\tau)\frac{\partial L}{\partial
x^{(i+j)}}\langle x, z  \rangle_n (\tau)\right)^{(j-1)}\eta^{(i)}(\tau)\right)\Bigg\vert_{\tau=b}=0.
$$
Let $I : = \left\{ i \in \{0, \ldots, n-1\} : x^{(i)}(b) \text{ is given}\right\}$.
Note that $k \not \in I$. For any $i \in \{0, \ldots, n-1\}$, if $i \in I$, then $\eta^{(i)}(b)=0$;
otherwise, for $i \ne k$, we restrict ourselves to those functions $\eta$ such that $\eta^{(i)}(b)=0$.
From the arbitrariness of $\eta^{(k)}(b)$, it follows that
$$
\sum_{j=1}^{n-k}(-1)^{j-1}\left( \lambda(\tau)\frac{\partial L}{\partial
x^{(k+j)}}\langle x, z  \rangle_n (\tau)\right)^{(j-1)}\Bigg\vert_{\tau=b}=0.
$$
This concludes the proof.
\end{proof}

\begin{remark}
If $x$ is a solution to problem $(P)$ without
any of the $2n$ boundary conditions \eqref{main_constrains},
then $x$ satisfies the generalized higher-order Euler--Lagrange equation \eqref{H-O_E-L_equations}
and $n$ transversality conditions \eqref{NBC(b)} and $n$ transversality conditions \eqref{NBC(a)}.
In general, for each boundary condition missing in \eqref{main_constrains}, there is a corresponding
natural boundary condition, as given by Theorem~\ref{Thm_Nat_Bound_Cond}.
\end{remark}

Next we remark that our generalized transversality conditions \eqref{NBC(b)} and \eqref{NBC(a)}
are generalizations of the classical transversality conditions for higher-order variational problems
(cf. $\psi^k = 0$, $k = 0,\ldots, n-1$, with $\psi^k$ given as in \cite[Section~5]{MR2098297}).

\begin{corollary}
Suppose that $x$ is a solution of problem $(P)$ with $L$ independent of $z$. Then
$x$ satisfies the classical Euler--Lagrange equation \eqref{ClassicalE-L}. Moreover,
\begin{enumerate}

\item If $x^{(k)}(b)$, $k \in \{0,\ldots, n-1\}$, is free, then
the natural boundary condition
\begin{equation*}
\sum_{j=1}^{n-k}(-1)^{j-1} \frac{d^{j-1}}{dt^{j-1}}
\left(\frac{\partial L}{\partial x^{(k+j)}}\right)\left(b, \dot{x}(b), \ldots, x^{(n)}(b)\right)=0
\end{equation*}
holds.

\item If $x^{(k)}(a)$ is free, $k \in \{0,\ldots, n-1\}$,
then the natural boundary condition
\begin{equation*}
\sum_{j=1}^{n-k}(-1)^{j-1} \frac{d^{j-1}}{dt^{j-1}}
\left(\frac{\partial L}{\partial x^{(k+j)}}\right)\left(a, \dot{x}(a), \ldots, x^{(n)}(a)\right)=0
\end{equation*}
holds.

\end{enumerate}
\end{corollary}


\section{Illustrative examples}
\label{section4_ex}

We illustrate the usefulness of our results with some examples
that are not covered by previous available results in the literature.
Let us consider the particular case of Theorem~\ref{main_Thm} with $n=2$.

\begin{corollary}
\label{cor:n2}
Let $z$ be a solution of $\dot{z}(t)=L\left(t,x(t),\dot{x}(t),\ddot{x}(t),z(t)\right)$,
$t\in [a,b]$, subject to the boundary conditions $z(a)=\gamma$,
$x(a)=\alpha_0$, $\dot{x}(a)=\alpha_1$,
$x(b)=\beta_0$, and $\dot{x}(b)=\beta_1$,
where $\gamma$, $\alpha_0$, $\alpha_1$, $\beta_0$,
and $\beta_1$, are given real numbers.
If $x$ is an extremizer of functional $z[x;b]$, then $x$ satisfies
the differential equation
\begin{multline}
\label{eq:EL2}
\frac{\partial L}{\partial x}\langle x, z \rangle_2(t)
+\frac{\partial L}{\partial z}\langle x, z \rangle_2(t)
\frac{\partial L}{\partial \dot{x}}\langle x, z  \rangle_2(t)
-\frac{d}{dt}\frac{\partial L}{\partial \dot{x}}\langle x, z \rangle_2 (t)
+ \left( \frac{\partial L}{\partial z}\langle x, z \rangle_2(t)\right)^2
\frac{\partial L}{\partial \ddot{x}}\langle x, z  \rangle_2 (t)\\
-2\frac{\partial L}{\partial z}\langle x, z  \rangle_2 (t)
\frac{d}{dt}\frac{\partial L}{\partial \ddot{x}}\langle x, z \rangle_2(t)
- \left(\frac{d}{dt}\frac{\partial L}{\partial z}\langle x, z \rangle_2(t)\right)
\frac{\partial L}{\partial \ddot{x}}\langle x, z \rangle_2 (t)
+ \frac{d^2}{dt^2}\frac{\partial L}{\partial \ddot{x}}\langle x, z \rangle_{2}(t) = 0,
\end{multline}
$t\in [a,b]$, where $\langle x, z \rangle_{2}(t)
= \left(t,x(t),\dot{x}(t),\ddot{x}(t),z(t)\right)$.
\end{corollary}

We now apply Corollary~\ref{cor:n2} to concrete situations.

\begin{example}
Let us consider the following Herglotz's
higher-order variational problem:
\begin{equation}
\label{eq:prb:del}
\begin{gathered}
z(1) \longrightarrow \min, \\
\dot{z}(t) = \ddot{x}^2(t) + z^2(t), \quad t \in [0,1], \quad z(0) = \frac{1}{2},\\
x(0) = 0, \quad \dot{x}(0) = 1, \quad x(1) = 1, \quad \dot{x}(1) = 1.
\end{gathered}
\end{equation}
For this problem, the necessary optimality condition \eqref{eq:EL2} asserts that
\begin{equation}
\label{eq:EL:ex}
x^{(4)}(t) - 4 z(t) x^{(3)}(t) + \left(4 z^2(t) - 2 \dot{z}(t)\right) x^{(2)}(t) = 0.
\end{equation}
Solving the system formed by \eqref{eq:EL:ex} and $\dot{z}(t) = \ddot{x}^2(t) + z^2(t)$,
subject to the given boundary conditions, gives the extremal
$$
x(t) = t, \quad z(t) = \frac{1}{2-t},
$$
for which $z(1) = 1$.
\end{example}

\begin{example}
Consider problem \eqref{eq:prb:del} with $z(0) = z_0$ free.
We show that such problem is not well defined.
Indeed, if a solution exists, we obtain the optimality system
\begin{equation}
\label{eq:EL:ex:zf}
\begin{cases}
x^{(4)}(t) - 4 z(t) x^{(3)}(t) + \left(4 z^2(t) - 2 \dot{z}(t)\right) x^{(2)}(t) = 0\\
\dot{z}(t) = \ddot{x}^2(t) + z^2(t)
\end{cases}
\end{equation}
subject to $x(0) = 0$ and $\dot{x}(0) = x(1) = \dot{x}(1) = 1$. It follows that
$$
x(t) = t, \quad z(t) = \frac{z_0}{1-z_0t},
$$
and we conclude that the problem has no solution:
the infimum is $-\infty$ obtained when $z_0 \rightarrow 1^+$.
\end{example}

\begin{example}
Consider now the following problem:
\begin{equation}
\label{eq:prb2:del}
\begin{gathered}
z(1) \longrightarrow \min, \\
\dot{z}(t) = \ddot{x}^2(t) + z(t), \quad t \in [0,1], \quad z(0) = 1,\\
x(0) = 0, \quad \dot{x}(0) = 1, \quad x(1) = 1, \quad \dot{x}(1) = 0.
\end{gathered}
\end{equation}
For problem \eqref{eq:prb2:del}, the necessary optimality condition \eqref{eq:EL2} asserts that
\begin{equation}
\label{eq:EL:ex2}
x^{(4)}(t) - 2 x^{(3)}(t) + x^{(2)}(t) = 0.
\end{equation}
Solving the system formed by \eqref{eq:EL:ex2} and $\dot{z}(t) = \ddot{x}^2(t) + z(t)$,
subject to the given boundary conditions, gives the extremal
$$
x(t) = \frac{(1-t) e^{t+1} + (2t - 1) e^t + (e - 3) e t - e + 1}{e^2 - 3 e + 1},
$$
$$
z(t) =
\frac{\left[ (1 + t^{2}) e^{t+2} - 2 (2 t^{2}+t + 2) e^{t+1}
+ (4 t^{2} +4 t + 5) e^{t} +e^{4} -6\,e^{3}+ 10\,e^{2}-2\,e-4\right]e^{t}}{\left(e^{2}-3\,e+1\right)^{2}},
$$
for which $z(1) = \frac{\left(e^2 - e - 4\right) e}{e^2 - 3 e + 1} \gtrsim 7,78$.
\end{example}

Our last example shows the usefulness of Theorem~\ref{Thm_Nat_Bound_Cond}.

\begin{example}
We now consider problem \eqref{eq:prb2:del} with $\dot{x}(1)$ free.
In this case, solving
$$
\begin{cases}
x^{(4)}(t) - 2 x^{(3)}(t) + x^{(2)}(t) = 0\\
\dot{z}(t) = \ddot{x}^2(t) + z(t)
\end{cases}
$$
subject to the boundary conditions
$z(0) = 1$, $x(0) = 0$, $\dot{x}(0) = 1$, $x(1) = 1$,
and the natural boundary condition \eqref{NBC(b)} for $n = 2$ and $k = 1$, that
in the present situation simplifies to $\ddot{x}(1) = 0$,
gives the extremal
$$
x(t) = t, \quad z(t) = e^t,
$$
for which $\dot{x}(1) = 1$ and $z(1) = e \lesssim 2,72$.
\end{example}


\section*{Acknowledgments}

This work was supported by FEDER funds through
COMPETE--Operational Programme Factors of Competitiveness
(``Programa Operacional Factores de Competitividade'')
and by Portuguese funds through the {\it Center for Research and Development
in Mathematics and Applications} (University of Aveiro) and the Portuguese
Foundation for Science and Technology (``FCT--Funda\c{c}\~{a}o para a Ci\^{e}ncia e a Tecnologia''),
within project PEst-C/MAT/UI4106/2011 with COMPETE number FCOMP-01-0124-FEDER-022690.
Torres was also supported by the FCT project PTDC/EEI-AUT/1450/2012,
co-financed by FEDER under POFC-QREN with COMPETE reference FCOMP-01-0124-FEDER-028894.
The authors are grateful to two anonymous referees for their valuable
comments and helpful suggestions.



\end{document}